\documentclass[11pt,a4paper]{amsart}

\makeatletter
\@namedef{subjclassname@2010}{%
 \textup{2010} Mathematics Subject Classification}
\makeatother

\usepackage[T1]{fontenc}
\usepackage{textcomp}

\usepackage[margin=1.2in]{geometry}

\usepackage{amssymb}

\usepackage{amsrefs}

\usepackage{xspace}

\theoremstyle{plain}
\newtheorem{thm}{Theorem}[section]
\newtheorem*{thm*}{Theorem}
\newtheorem*{lem*}{Lemma}

\newtheorem{prop}[thm]{Proposition}
\newtheorem{lem}[thm]{Lemma}

\theoremstyle{definition}
\newtheorem{defn}{Definition}[section]

\DeclareMathOperator{\seqsep}{sep}

\begin{document}

\title{Property $(\beta)$ and uniform quotient maps}

\author{Vegard Lima}
\author{N. Lovasoa Randrianarivony}

\address{Aalesund, University College, Postboks 1517, 6025 Aalesund, Norway}
\email{Vegard.Lima@gmail.com}

\address{Department of Mathematics and Computer Science\\
Saint Louis University \\
St. Louis, MO  63103, USA}
\email{nrandria@slu.edu}


\thanks{The second author was supported by NSF grant DMS--0915349.}

\maketitle

\begin{abstract}
In 1999, Bates, Johnson, Lindenstrauss, Preiss and Schechtman asked whether a Banach space that is a uniform quotient of $\ell_p$, $1 < p \neq 2 < \infty$, must be isomorphic to a linear quotient of $\ell_p$.  We apply the geometric property $(\beta)$ of Rolewicz to the study of uniform and Lipschitz quotient maps, and answer the above question positively for the case $1<p<2$.  We also give a necessary condition for a Banach space to have $c_0$ as a uniform quotient.
\end{abstract}

\section{Introduction}

Consider two metric spaces $X$ and $Y$, and a map $T:X \to Y$.  The following definitions are taken for example from \cite[Definition 11.10]{BenyaminiLindenstrauss2000}.  The map $T$ is called \emph{co-uniformly continuous} if for every $\varepsilon >0$ there exists $\delta >0$ such that for every $x$ in $X$,
\begin{equation*}
B(Tx, \delta) \subseteq T\left ( B(x,\varepsilon) \right).
\end{equation*}
(Throughout the present article, all balls are considered closed.)  If $\delta$ can be chosen to be linear in $\varepsilon$,
then $T$ is called \emph{co-Lipschitz}.
The map $T$ is called a \emph{uniform quotient} (resp. \emph{Lipschitz quotient}) if it is both uniformly continuous and co-uniformly continuous (resp. Lipschitz and co-Lipschitz).  If $T$ is also surjective, then $Y$ is called a \emph{uniform quotient} (resp. \emph{Lipschitz quotient}) of $X$.

\indent

Bates, Johnson, Lindenstrauss, Preiss and Schechtman have shown in \cite{BatesJohnsonLindenstraussPreissSchechtman1999} that the linear quotients of $L_p[0,1]$ are isomorphically the only possible Banach spaces that can be a uniform quotient of $L_p[0,1]$ for $1<p<\infty$.  Using a property called UAAP they showed that any Banach space that is a uniform quotient of a superreflexive Banach space $X$ is a linear quotient of some ultrapower of $X$.  This gives an isomorphic characterization of the Banach spaces which are uniform quotients of $L_p[0,1]$ ($1<p<\infty$) but left the gap open for $\ell_p$.

\indent

In \cite{JohnsonLindenstraussPreissSchechtman2002}, Johnson, Lindenstrauss, Preiss and Schechtman show among other things that $\ell_2$ is not a Lipschitz quotient of $\ell_p$ for $2<p<\infty$.  The argument is based on a differentiation technique and a comparison of the moduli of asymptotic uniform convexity and smoothness of $\ell_p$ and $\ell_2$ which does not extend to the case of $1<p<2$.  

\indent

In the present article, we show that $\ell_2$ cannot be a uniform quotient of $\ell_p$ for any $1<p<2$.  This shows that the only Banach spaces that are uniform quotients of $\ell_p$ for $1<p<2$ are isomorphically the linear quotients of $\ell_p$ $(1<p<2$) themselves.  Our result relies on a geometric property introduced by Rolewicz called property ($\beta$).  In Section \ref{beta} we present the ingredients about property $(\beta)$ that we need.  In Section \ref{unifcounif} we present lemmas about uniformly and co-uniformly continuous maps.  Although most of the material in these sections can be found elsewhere, we gather them here for completeness and also to set our notation.  The last section will be for our main result.

\section{Property $(\beta)$}\label{beta}

In  \cite{Rolewicz1987too}, Rolewicz introduces property $(\beta)$, which he has shown in \cite{Rolewicz1987} to be one generalization of uniform convexity in the sense that a uniformly convex Banach space must be $(\beta)$.  In \cite{Kutzarova1991}, Kutzarova gave the following characterization of property $(\beta)$, which we will adopt as the definition throughout the present article.

\begin{prop}
Let $X$ be a Banach space, and denote by $B_X$ its unit ball.  Then $X$ has the property $(\beta)$ of Rolewicz if and only if for every $\varepsilon>0$ there exists $\delta>0$ such that for every element $x \in B_X$ and every sequence $(x_n)_n \subseteq B_X$ with $\seqsep (\{x_n\})\geq \varepsilon$, there exists an index $i$ such that
$$\left \|\frac{x_i+x}{2}\right\|\leq 1-\delta.$$
\end{prop}
Here the separation of the sequence is defined by $\seqsep(\{x_n\})=\inf \{\|x_n-x_m\|: n\neq m\}$.

\indent

Ayerbe, Dom\'inguez Benavides, and Cutillas define three moduli for the property $(\beta)$ in \cite{AyerbeDominguez_BenavidesCutillas1994}, and compute them for the space $\ell_p$ ($1<p<\infty$).  The following is what we will use.

\begin{prop}[see \cite{AyerbeDominguez_BenavidesCutillas1994}]
Let $X$ be a Banach space, and consider the following modulus $R''_X:[0,a] \to [0,1]$ where $a \in [1,2)$ depends on the space $X$:
$$R''_X(\varepsilon)=1-\sup\left\{\inf\left\{\frac{\|x+x_n\|}{2}, n\in \mathbb{N}\right\}: (x_n)_n \subseteq B_X, x\in B_X, \seqsep(\{x_n\})\geq \varepsilon\right\},$$
then for $1<p<\infty$, $R''_{\ell_p}:[0,2^{1/p}]\to[0,1]$ is given by
$$R''_{\ell_p}(\varepsilon)=1-\frac{1}{2}\left\{\frac{\varepsilon^p}{2}+\left(\left(1-\frac{\varepsilon^p}{2}\right)^{\frac{1}{p}}+1\right)^p\right\}^{\frac{1}{p}}.$$
\end{prop}

\indent

This allows us to make the following computation.

\begin{lem}\label{geometric_lemma}
  Let $1 < p < \infty$. Let $0<\varepsilon \leq 2^{1/p}$ and $0<\delta<2$.
  Assume that $z \in B_{\ell_p}$ and $(z_n) \subset B_{\ell_p}$
  with $\|z - z_n\| > 2 - \delta$ for all $n \in \mathbb{N}$
  and $\seqsep(\{z_n\}) \ge \varepsilon$.  Then $\varepsilon \le C_p\,\delta^{1/p}$ for some positive
  constant $C_p$ depending only on $p$.
\end{lem}

\begin{proof}
  We have
  \begin{multline*}
    2 - 2R''_{\ell_p}(\varepsilon) =
    \sup\left\{\inf\left\{
        \|x_n + x\| : n \in \mathbb{N} \right\}
      : (x_n)_n \subset B_{\ell_p}, x \in B_{\ell_p}, \seqsep(\{x_n\}) \ge
      \varepsilon \right\} \\
    \ge \inf\{\|z_n + (-z)\| : n \in \mathbb{N}\}
    = \inf\{\|z - z_n\| : n \in \mathbb{N}\}
    > 2 - \delta.
  \end{multline*}

   So we have
  \begin{equation*}
  \begin{split}
  (2-\delta)^p&< \left(2 - 2R''_{\ell_p}(\varepsilon)\right)^p\\
  &= \frac{\varepsilon^p}{2}+\left(\left(1-\frac{\varepsilon^p}{2}\right)^{1/p}+1\right)^p\\
  &\leq \frac{\varepsilon^p}{2}+\left(\left(1-\frac{\varepsilon^p}{2p}\right)+1\right)^p\\
  &= \frac{\varepsilon^p}{2}+2^p\left(1-\frac{\varepsilon^p}{4p}\right)^p\\
  \end{split}
  \end{equation*}
  This is because $\displaystyle (1-t)^{1/p}<1-\frac{t}{p}$ for all
  $0<t\leq 1$.  Next notice that $\displaystyle
  \frac{\varepsilon^p}{4p}\leq \frac{1}{2p}<\frac{1}{2}$, and that for
  $\displaystyle 0<t<\frac{1}{2}$ we have
 $$(1-t)^p<1-\frac{2(2^p-1)}{2^p}t$$ 
 by convexity of the function $\varphi(t)=(1-t)^p$.  We then have
 \begin{equation*}
  \begin{split}
  (2-\delta)^p&< \frac{\varepsilon^p}{2}+2^p\left(1-\frac{2(2^p-1)\varepsilon^p}{4p2^p}\right)\\
  &=2^p\left(1-\frac{(2^p-1)}{p2^{p+1}}\varepsilon^p+\frac{1}{2^{p+1}}\varepsilon^p\right)\\
  \end{split}
  \end{equation*} 
  so that
  $$\left(1-\frac{\delta}{2}\right)^p<1-\frac{(2^p-1)}{p2^{p+1}}\varepsilon^p+\frac{1}{2^{p+1}}\varepsilon^p.$$
  
  Next, we note that
  $$1-p\frac{\delta}{2}<\left(1-\frac{\delta}{2}\right)^p,$$
  giving us
  $$\left(\frac{(2^p-1)}{p2^{p+1}}-\frac{1}{2^{p+1}}\right)\varepsilon^p <\frac{p}{2}\delta.$$ 
   
   Finally, we check that $\displaystyle \left(\frac{(2^p-1)}{p2^{p+1}}-\frac{1}{2^{p+1}}\right)>0$, which comes down to checking that $2^p>1+p$, which is true for any $p>1$.      
    
\end{proof}

Geometrically the above lemma tells us that
if we have a ``fork'' or a ``chinese fan'' in the unit ball, with
the line segment $[z,0]$ as the handle
and the points $z_n$ as the tips, then
the tips cannot be separated too much if the
length of the fork is almost the diameter of the unit ball.

\section{Uniformly and co-uniformly continuous maps}\label{unifcounif}

Let $T:X\to Y$ be a map between two metric spaces $X$ and $Y$.  Denote by $\Omega$ the modulus of continuity of $T$, namely
\begin{equation}\label{modofcont}
\Omega(t):=\sup ~\{d(Tx,Tx'), d(x,x') \leq t \}.
\end{equation}
Then $\Omega$ is nondecreasing, and $T$ is uniformly continuous if and only if $~\Omega(t) \to 0$ as $t \to 0$.  

\indent

The following lemma is classical, see for example \cite[Proposition 1.11]{BenyaminiLindenstrauss2000}.	
\begin{lem}[``Lipschitz for large distances'' principle]\label{lip4lrgdist}
Assume that the metric space $X$ is a Banach space, (or in general that $X$ is metrically convex).  Then for any $M\geq 0$ and any $d> 0$, we have
$$\Omega(M) \leq \max \left \{\Omega(d), 2\frac{\Omega(d)}{d}M \right \}.$$
As a result of this and the fact that $\Omega(t) \to 0$ as $t\to 0$, a uniformly continuous $T$ will satisfy $\Omega(M) <+\infty$ for any $M\geq 0$.
\end{lem}

There is also a ``co-Lipschitz for large distances'' principle.  The proof can be found in \cite[Lemma 11.11]{BenyaminiLindenstrauss2000} for example, but is included here for completeness.

\begin{lem}[``co-Lipschitz for large distances'' principle]\label{colip4lrgdist}
Let $T:X \to Y$ be co-uniformly continuous and assume that $Y$ is a Banach space (or more generally that $Y$ is metrically convex).  Then for every $d>0$, there exists $c>0$ such that
\begin{equation}\label{c(d)}
r\geq d ~\Rightarrow~ \forall x\in X, ~~B(Tx, cr) \subseteq T\left(B(x,r)\right).
\end{equation}
\end{lem}

	\begin{proof}
	Let $d>0$, and let $\delta=\delta\left (\frac{d}{2} \right)$ be the $\delta$ given by $\varepsilon=\frac{d}{2}$ from the definition of co-uniform continuity, so that we have for every $x$ in $X$,
	\begin{equation}\label{delta}
	B(Tx,\delta) \subseteq T\left(B\left(x,\frac{d}{2}\right)\right).
	\end{equation}
	
	Let $r\geq d$ and write $r=n\frac{d}{2}+d'$ with $n \in \mathbb{N}$ and $0<d'\leq \frac{d}{2}$.  
	
	Let $x\in X$ and $y\in Y$ be such that $\|y-Tx\| \leq n\delta$.  Let us divide the segment $[Tx,y]$ into $n$ segments of equal length $\frac{\|y-Tx\|}{n} \leq \delta$:
	$$Tx=y_0,~y_1, ~y_2, ~\cdots, ~y_n=y.$$
	Then from using (\ref{delta}) inductively we get points $x_1, ~x_2, ~\cdots, ~x_n \in X$ such that
	\begin{equation*}
	\begin{cases}
	Tx_i=y_i \text{ for all } 1\leq i\leq n\\
	\\
	\|x_1-x\|\leq \frac{d}{2},\\
	\\
	\|x_i-x_{i-1}\| \leq \frac{d}{2} \text{ for all } 2\leq i \leq n.\\
	\end{cases}
	\end{equation*}
	
	As a result, the point $x':=x_n$ satisfies $\|x'-x\| \leq n\frac{d}{2} \leq r$ and $Tx'=y$.
	
	Now, from the definition of $n$ and $d'$, we have that
	$$\frac{r}{d}=\frac{n\frac{d}{2}+d'}{d}\leq \frac{n}{2}+\frac{1}{2}\leq n.$$	
	
	This, together with the above argument, shows that if $\|y-Tx\| \leq \frac{r}{d}\delta$ then we get $\|y-Tx\| \leq n\delta$ so that $y=Tx'$ for some $x'\in B(x,r)$.  This means that we can take
	\begin{equation}\label{lowerbound}
	c= \frac{\delta\left(\frac{d}{2}\right)}{d}.
	\end{equation}
	\end{proof}
	
\indent
	
The next lemma looks at the interaction of the facts that both $X$ and $Y$ are Banach spaces, and $T:X\to Y$ is both uniformly and co-uniformly continuous.
	
\begin{lem}
Let $X$ and $Y$ be Banach spaces, and let $T:X \to Y$ be both uniformly and co-uniformly continuous.  For $d>0$, denote $c_d$ the supremum of all $c>0$ such that
\begin{equation}\label{c(d)attained}
\left(r \geq d,~x\in X,~\|y-Tx\|<cr\right)~\Rightarrow ~\left(y=Tx' \text{ for some }x'\in B(x,r)\right).
\end{equation}
Then, 
	\begin{enumerate}
	\item[(i)] $c_d$ is attained,
	\item[(ii)] $0<c_d<+\infty$ for all $d>0$,
	\item[(iii)] $\{c_d\}_{d>0}$ is nondecreasing, and the limit $\displaystyle C=\lim_{d\to \infty} c_d$ is finite.
	\end{enumerate}

\end{lem}
	\begin{proof}
	\mbox{ }
	\begin{enumerate}
	\item[(i)] The fact that $c_d$ is attained comes from the one strict inequality in equation (\ref{c(d)attained}).
	\item[(ii)] For any $\varepsilon>0$, any $r \ge d$ and any $x\in X$,
	$$B(Tx,(c_d-\varepsilon)r) \subseteq T(B(x,r)) \subseteq B(Tx,\Omega(r)).$$
	In particular for $r = d$, this and equation (\ref{lowerbound}) gives us that
	$$\frac{\delta\left(\frac{d}{2}\right)}{d}\leq c_d \leq \frac{\Omega(d)}{d}.$$
	\item[(iii)] It's trivial that $\{c_d\}_{d>0}$ is nondecreasing.  As in (ii) we also get $c_d\cdot r \leq \Omega(r)$ for $r\geq d$, and from Lemma \ref{lip4lrgdist} we get
	$$c_d \leq \frac{\Omega(r)}{r}
	  \leq \max \left\{ \frac{\Omega(1)}{r},2\Omega(1)\right\}\leq 2\Omega(1)<+\infty.$$	
	\end{enumerate}
	\end{proof}

\section{Main result}

\begin{thm}\label{main}
$\ell_q$ cannot be a uniform quotient of $\ell_p$ for $1<p<q<\infty$.
\end{thm}

Note that the case $q=2$ is the critical case here, the other cases are already proven by the result of \cite[Theorem 3.5]{BatesJohnsonLindenstraussPreissSchechtman1999}.  In fact, $\ell_q$ is not isomorphic to a linear quotient of $L_p[0,1]$ for $1<p<q<\infty$, $q\neq 2$, see \cite[Theorem 6.4.19]{AlbiacKalton2006}.  

	\begin{proof}
	Assume for contradiction that there is a $T : \ell_p \to \ell_q$ that is both uniformly and co-uniformly continuous.  We will adopt all the notation of the previous preliminary section for $T$.
	
	We have already seen that $c_d$ is non-decreasing and converges to $C\in (0,+\infty)$ as $d\to +\infty$.  For $0<\varepsilon<1$ pick $d_0>0$ so that 
	$$C-\varepsilon < c_{\frac{d_0}{3}} \leq C < C+\varepsilon.$$
	
	From the definition of $c_{d_0}$ as a supremum, and since $C+\varepsilon>c_{d_0}$, we get that there exist $z_\varepsilon \in \ell_p$, $R = r_\varepsilon \geq d_0$, and $y_\varepsilon \in \ell_q$ so that
  $\|y_\varepsilon - Tz_\varepsilon\| < (C+\varepsilon)R$
  yet $\|x-z_\varepsilon\| > R$ for all
  $x \in \ell_p$ with $Tx = y_\varepsilon$.  Note that $c_{d_0}R \leq \|Tz_\varepsilon - y_\varepsilon\|$
  since otherwise the fact that $R\geq d_0$ would imply $y_\varepsilon \in T(B(z_\varepsilon,R))$.  
    
  	\indent
  
  	Set $D := \|Tz_\varepsilon - y_\varepsilon\|$.
  Divide the line segment $[Tz_\varepsilon,y_\varepsilon]$
  into three segments of equal length:
  $[Tz_\varepsilon,m]$, $[m,M]$ and $[M,y_\varepsilon]$,
  \begin{equation*}
    c_{d_0} \frac{R}{3} \leq \|Tz_\varepsilon - m\|
    = \|m - M\| = \|M - y_\varepsilon\| = \frac{D}{3}
    < (C+\varepsilon) \frac{R}{3}.
  \end{equation*}
  
  	Now, note that 
	\begin{equation*}
	\|Tz_\varepsilon-m\| =\frac{D}{3} 
	< (C+\varepsilon)\frac{R}{3}
	=c_{\frac{d_0}{3}}\cdot \left(\frac{C+\varepsilon}{c_{\frac{d_0}{3}}}\right)\frac{R}{3}.	\end{equation*}
	
	To simplify the reading, we are going to introduce the notation $\rho_\varepsilon= \left(\frac{C+\varepsilon}{c_{\frac{d_0}{3}}}\right)\frac{R}{3}$.  We have	
	\begin{equation*}
	\rho_\varepsilon
	\geq \left(\frac{C+\varepsilon}{C}\right)\frac{R}{3} 
	\geq \frac{R}{3}\geq \frac{d_0}{3},
	\end{equation*}	
	so by definition of $\displaystyle c_{\frac{d_0}{3}}$ we get that $m = Tz$ for some $z \in \ell_p$ with
  $\|z_\varepsilon - z\| \le \rho_\varepsilon$.
  
  	\indent

  	Without loss of generality we will assume that
  $m= 0$ and that $z = 0$.  This allows us to write in particular that $y_\varepsilon=2M$.
  
  	Denote by $(e_n)_n$ the unit vector basis of $\ell_q$.  Let $M_N$ be a truncation of $M$  supported on $[1,N]$ so that $\|M-M_N\|< \varepsilon \frac{D}{3}$, and for $n>N$ set
	 \begin{equation*}
    y_n := \varepsilon^{1/q} \frac{D}{3} e_n
    + (1-\varepsilon)^{1/q} M_N.
  \end{equation*}
  	Then
  	\begin{equation*}
  	\begin{split}
    	\|y_n\|^q & = \varepsilon\left(\frac{D}{3}
    \right)^q + (1-\varepsilon)\|M_N\|^q\\
    &\leq \varepsilon\left(\frac{D}{3}
    \right)^q + (1-\varepsilon)\left(\frac{D}{3}\right)^q\\
        &= \left(\frac{D}{3}\right)^q.\\
    	\end{split}
  	\end{equation*}
	As a result, we also have $y_n=Tz_n$ for some $z_n \in \ell_p$ with $\|z_n-0\| \leq \rho_\varepsilon$.
	
	On the other hand,
	\begin{equation*}
	\begin{split}
	\|y_n-y_{\varepsilon}\|^q
	&=\left \|\varepsilon^{1/q} \frac{D}{3} e_n
    + (1-\varepsilon)^{1/q} M_N-2M\right \|^q\\
    	& =\left \|\varepsilon^{1/q} \frac{D}{3} e_n+2(M_N-M)
    + (1-\varepsilon)^{1/q} M_N-2M_N\right \|^q\\
    	&=\left\|\varepsilon^{1/q} \frac{D}{3} e_n+2(M_N-M)\right\|^q+\left(2-(1-\varepsilon)^{1/q}\right)^q\|M_N\|^q\\	
    	&\leq\left(\varepsilon^{1/q}\frac{D}{3}+2\|M_N-M\|\right)^q+\left(2-(1-\varepsilon)^{1/q}\right)^q\|M_N\|^q\\
	&\leq \left(\varepsilon^{1/q}\frac{D}{3}+2\varepsilon\frac{D}{3}\right)^q+\left(2-(1-\varepsilon)^{1/q}\right)^q\left(\frac{D}{3}\right)^q\\	
	&\leq 3^q\varepsilon\left(\frac{D}{3}\right)^q+\left(2-(1-\varepsilon)^{1/q}\right)^q\left(\frac{D}{3}\right)^q\text{ since } \varepsilon<1\\	
	& < 3^q\varepsilon \left(\frac{D}{3}\right)^q+\left(2-(1-\varepsilon)\right)^q\left(\frac{D}{3}\right)^q\text{ since } (1-\varepsilon)^{1/q}>1-\varepsilon\\
	&= 3^q\varepsilon \left(\frac{D}{3}\right)^q+(1+\varepsilon)^q\left(\frac{D}{3}\right)^q\\	&<(1+2\cdot3^q\varepsilon)\left(\frac{D}{3}\right)^q\text{ since } (1+\varepsilon)^q<1+(2^q-1)\varepsilon \text{ by convexity.}\\
	\end{split}
	\end{equation*}
	Taking $q$th roots, and using the fact that $(1+t)^{1/q}<1+\displaystyle \frac{t}{q}$ for $t>0$, we then get
	\begin{equation*}
	\begin{split}
	\|y_n-y_\varepsilon\| 
	&< \left(1+\frac{2\cdot3^q}{q}\varepsilon\right)\frac{D}{3}\\
	&< \left(1+\frac{2\cdot3^q}{q}\varepsilon\right)(C+\varepsilon)\frac{R}{3}\\
	&=c_{\frac{d_0}{3}}\left(1+\frac{2\cdot3^q}{q}\varepsilon\right)\rho_\varepsilon\\
	\end{split}
	\end{equation*}
	with
	$$\left(1+\frac{2\cdot3^q}{q}\varepsilon\right)\rho_\varepsilon\geq \rho_\varepsilon\geq \frac{R}{3}\geq \frac{d_0}{3},$$
	so again by definition of $c_{\frac{d_0}{3}}$, we get a point $x_n \in \ell_p$ with $Tx_n=y_\varepsilon$ and $\|x_n-z_n\| \leq \left(1+\frac{2\cdot3^q}{q}\varepsilon\right)\rho_\varepsilon$.

  \indent
  
  In summary we have $\|z_\varepsilon - x_n\| > R$, $\|z_\varepsilon\| \le \rho_\varepsilon$, and $\|z_n\| \le \rho_\varepsilon$, while
  $$\|z_n - x_n\| \le \left(1+\frac{2\cdot3^q}{q}\varepsilon\right)\rho_\varepsilon.$$
  
  Together these imply that
  \begin{equation*}
  \begin{split}
    \|z_\varepsilon - z_n\| &\ge \|z_\varepsilon - x_n\| - \|z_n -
    x_n\|\\
    &> R - \left(1+\frac{2\cdot3^q}{q}\varepsilon\right)\rho_\varepsilon\\
    &= \left(\frac{3c_{\frac{d_0}{3}}}{C+\varepsilon} - 1-\frac{2\cdot3^q}{q}\varepsilon\right)\rho_\varepsilon. \\
  \end{split}
  \end{equation*}
  
  \indent
  
  On the other hand, we have for $n \neq m$
  \begin{equation*}
    \|y_n - y_m\| = 
    \varepsilon^{1/q} \frac{D}{3}\|e_n-e_m\| = 2^{1/q}\varepsilon^{1/q}\frac{D}{3}>2^{1/q}\varepsilon^{1/q}(C-\varepsilon)\frac{R}{3}.
  \end{equation*}
  Now, we could very well have started with $\varepsilon$ small enough so that $C-\varepsilon>\displaystyle \frac{C}{2}$, and then chosen $d_0$ large enough so that 
  $$2^{1/q}\varepsilon^{1/q}\frac{C}{2}\frac{R}{3}\geq 2^{1/q}\varepsilon^{1/q}\frac{C}{2}\frac{d_0}{3}>\Omega(1)$$
  where $\Omega$ is the modulus of continuity of $T$.  With these choices, the uniform continuity of $T$ will give
  $$\Omega(1)< \|y_n-y_m\|=\|Tz_n-Tz_m\|\leq \Omega(\|z_n-z_m\|).$$
  Since $\Omega$ is nondecreasing, it then must follow that 
  $\|z_n-z_m\|>1$,
  and hence the Lipschitz condition for $T$ for distances larger than $1$ will give
  $$\|y_n-y_m\|\leq 2\Omega(1)\|z_n-z_m\|.$$
    
  \indent
  
  All these put together will give us that
  \begin{equation*}
  \begin{split}
  \|z_n-z_m\| &\geq \frac{1}{2\Omega(1)}2^{1/q}\varepsilon^{1/q}\frac{D}{3}\\
  & \geq \frac{1}{2^{1-1/q}\Omega(1)}\varepsilon^{1/q}c_{d_0}\frac{R}{3}\\
  &=\frac{1}{2^{1-1/q}\Omega(1)}\varepsilon^{1/q}\left(c_{d_0}\cdot\frac{c_{\frac{d_0}{3}}}{C+\varepsilon}\right)\rho_\varepsilon\\
  \end{split}
  \end{equation*}
  while $\|z_\varepsilon\|,\|z_n\|\leq\rho_\varepsilon.$
  Assuming that $\varepsilon$ is small enough so that the quantification of property $(\beta)$ in Lemma \ref{geometric_lemma} applies, we get that
  \begin{equation*}
\frac{1}{2^{1-1/q}\Omega(1)}\varepsilon^{1/q}\left(c_{d_0}\cdot\frac{c_{\frac{d_0}{3}}}{C+\varepsilon}\right)
\leq C_p\left[2-\frac{3c_{\frac{d_0}{3}}}{C+\varepsilon}+1+\frac{2\cdot3^q}{q}\varepsilon\right]^{1/p} .
  \end{equation*}
  Since $\displaystyle C-\varepsilon<c_{\frac{d_0}{3}}\leq c_{d_0}\leq C$, we get
  \begin{equation*}
\frac{1}{2^{1-1/q}\Omega(1)}\varepsilon^{1/q}\frac{(C-\varepsilon)^2}{C+\varepsilon}
\leq C_p\left[3-\frac{3C}{C+\varepsilon}+\frac{2\cdot3^q}{q}\varepsilon\right]^{1/p}, 
  \end{equation*}
  then using the fact that $\displaystyle 1-\frac{2\varepsilon}{C} \leq \frac{C}{C+\varepsilon}$, we get
  \begin{equation*}
  \begin{split}
	\frac{1}{2^{1-1/q}\Omega(1)}\varepsilon^{1/q}\frac{(C-\varepsilon)^2}{C+\varepsilon}
	&\leq C_p\left[3-3\left(1-\frac{2\varepsilon}{C}\right)+\frac{2\cdot3^q}{q}\varepsilon\right]^{1/p} \\
	&=C_p\left[\frac{6}{C}+\frac{2\cdot3^q}{q}\right]^{1/p}\varepsilon^{1/p}\\
	\end{split}
  \end{equation*}
  Hence
  $$\frac{1}{2^{1-1/q}\Omega(1)}\frac{(C-\varepsilon)^2}{C+\varepsilon}
  \leq C_p\left[\frac{6}{C}+\frac{2\cdot3^q}{q}\right]^{1/p}\varepsilon^{1/p-1/q}.$$
  Letting $\varepsilon \to 0$ gives a contradiction since $p<q$.	
	\end{proof}

\indent

We make use of the Lipschitz and co-Lipschitz for large distances principles in the above argument.  In the Lipschitz quotient case where we do not need to require metric convexity, a similar argument with slightly easier computations gives us the following.

\begin{thm}
$\ell_q$ is not a Lipschitz quotient of any subset of $\ell_p$ for any $1<p<q<\infty$.  
\end{thm}

In \cite{MendelNaorExtendedAbstract2008}, Mendel and Naor show the same result for $2\leq p<q<\infty$ through the use of a technique called Markov convexity.  

\indent

Our technique also gives us the following result.  

\begin{thm}
$c_0$ cannot be a uniform quotient of (or a Lipschitz quotient of a subset of) a Banach space with property $(\beta)$.
\end{thm}

The result of \cite[Theorem 3.5]{BatesJohnsonLindenstraussPreissSchechtman1999} gives that $c_0$ is not a uniform quotient of a superreflexive Banach space, but as shown in \cite{Kutzarova1989} and \cite{MontesinosTorregrosa1992} there are Banach spaces that satisfy property $(\beta)$ but are not superreflexive.

	\begin{proof}
	Let $X$ be a Banach space with property $(\beta)$.  We will prove that $c_0$ cannot be a uniform quotient of $X$.  The argument for the Lipschitz case is similar.
	
	Assume for contradiction that $T:X \to c_0$ is a uniform quotient map.  We follow exactly the same proof as for Theorem \ref{main} until the choice of $y_n$.  Instead, we set
	$$y_n:=\frac{D}{3}e_n+M_N.$$
	Then we have 
	\begin{equation*}
	\begin{split}
	\|y_n-y_{\varepsilon}\|
	&=\left \|\frac{D}{3} e_n+ M_N-2M\right \|\\
    	& =\left \|\frac{D}{3} e_n+2(M_N-M)-M_N\right \|\\
    	&=\max \left\{\left\|\frac{D}{3} e_n+2(M_N-M)\right\|, \|M_N\|\right\}\\
	&< \max \left \{ \left(\frac{D}{3}+2\varepsilon\frac{D}{3}\right),\frac{D}{3}\right\}\\	
	&=(1+2\varepsilon)\frac{D}{3}.
	\end{split}
	\end{equation*}
	Choosing $z_n$ (and $x_n$) similarly as in the proof of Theorem \ref{main}, we then have that $\|z_\varepsilon\|, \|z_n\| \leq \rho_\varepsilon$, while
	$$\|z_n-z_\varepsilon\|\geq \left(\frac{3c_{\frac{d_0}{3}}}{C+\varepsilon} - 1-2\varepsilon\right)
   \rho_\varepsilon ,$$	
	and
	$$\|z_n-z_m\| 
	\geq \frac{1}{2\Omega(1)} \|y_n-y_m\|
	\geq \frac{1}{2\Omega(1)}\left(c_{d_0}\cdot\frac{c_{\frac{d_0}{3}}}{C+\varepsilon}\right)\rho_\varepsilon.$$	
	Since 
	$$\frac{1}{2\Omega(1)}\left(c_{d_0}\cdot\frac{c_{\frac{d_0}{3}}}{C+\varepsilon}\right)
	\geq \frac{1}{2\Omega(1)}\left(\frac{(C-\varepsilon)^2}{C+\varepsilon}\right)
	\geq \frac{1}{2\Omega(1)}\cdot\frac{C}{2}>0,$$ 
	the property $(\beta)$ of $X$ will give us a number $\delta>0$ independent of $\varepsilon$, and an index $i$ such that
	$$\|z_i-z_\varepsilon\| \leq (2-2\delta)\rho_\varepsilon.$$
	This then gives us
	$$\frac{3(C-\varepsilon)}{C+\varepsilon} - 1-2\varepsilon\leq 2-2\delta,$$
	a contradiction as $\varepsilon \to 0$. 
	\end{proof}

\indent

As a corollary to Theorem \ref{main} we get the following characterization of all Banach spaces that are uniform quotients of $\ell_p$ for $1<p<2$.

\begin{thm}
If a Banach space $X$ is a uniform quotient of $\ell_p$ for $1<p<2$, then $X$ has to be linearly isomorphic to a linear quotient of $\ell_p$.
\end{thm}
	\begin{proof}
	Bates, Johnson, Lindenstrauss, Preiss and Schechtman assert in \cite{BatesJohnsonLindenstraussPreissSchechtman1999} that $X$ has to be linearly isomorphic to a linear quotient of $L_p[0,1]$.  A result of Johnson and Odell \cite{JohnsonOdell1974} shows that if $\ell_2$ is not isomorphic to a linear quotient of $X$, then $X$ will be isomorphic to a linear quotient of $\ell_p$.
	
	But $\ell_2$ cannot be a linear quotient of $X$ since the composition of such a linear quotient map and the uniform quotient from $\ell_p$ to $X$ would give a uniform quotient from $\ell_p$ to $\ell_2$, contradicting Theorem \ref{main}.
	\end{proof}

\section{Acknowledgement}
We thank William B. Johnson for his comments over the first draft of this article.

\end{document}